\documentclass[11pt]{article}

\usepackage[margin=1in]{geometry}

\usepackage{graphics}
\usepackage{graphicx}
\graphicspath{{figures/}}
\usepackage{mathtools}

\usepackage{amsmath,amsthm,amsfonts}
\usepackage{latexsym}
\usepackage{amssymb}

\usepackage{bbm}

\usepackage{float}

\usepackage{enumitem}

\usepackage{color}

\usepackage{hyperref}

\usepackage{subcaption}

\usepackage{cite}

\usepackage{xspace}

\usepackage{algorithm}
\usepackage{algpseudocode}

\newtheorem{theorem}{Theorem}[section]

\theoremstyle{definition}
\newtheorem{definition}{Definition}[section]
\newtheorem{exmp}{Example}[section]

\theoremstyle{remark}
\newtheorem{remark}{Remark}[section]

\DeclareMathOperator{\Mod}{Mod}
\DeclareMathOperator{\Adm}{Adm}
\DeclareMathOperator{\supp}{supp}

\newcommand{\pB}{\textbf{B}\xspace}
\newcommand{\pE}{\textbf{E}\xspace}

\author{Nathan Albin, Kapila Kottegoda and Pietro Poggi-Corradini\\
\footnotesize{Department of Mathematics, Kansas State University, Manhattan, KS}}
\title{Spanning tree modulus for secure broadcast games\footnote{This material is based upon work supported by the National Science Foundation under Grant No.~1515810.}}

\begin{document}
	\maketitle
	\begin{abstract}
	The theory of $p$-modulus provides a general framework for quantifying the richness of a family of objects on a graph.  When applied to the family of spanning trees, $p$-modulus has an interesting
	probabilistic interpretation.  In particular, the $2$-modulus problem in this case has been
	shown to be equivalent to the problem of finding a probability distribution on spanning trees
	that utilizes the edges of the graph as evenly as possible.  In the present work, we use this fact to produce a game-theoretic interpretation of modulus by employing modulus to solve a secure broadcast game.
	\end{abstract}

	\section{Introduction}\label{sec:introduction}

	\subsection{Secure broadcast games}\label{sec:game-intro}
	A communication network is often modeled as a collection of nodes and
	links.     For example, in a wired computer network, nodes can represent
	personal computers and links can represent transmission media such as
	Ethernet cables. Mathematically, such a network is often modeled by a
	graph $G = (V,E)$, where $V$ is the set of vertices corresponding to the
	nodes and $E$ is the set of edges corresponding to the links.

	A (one-to-all) broadcast in a network is a method of sending a message
	from a source node to all other nodes in the network by utilizing a subset
	of the network links.  If the goal is to use as few links as possible, then the
	set of links used will form a spanning tree of the graph $G$.

	The present paper concerns a simple secure broadcast game introduced and analyzed in~\cite{gueye2010design}.  Although many modified versions of this game exist (e.g.,~\cite{gueye2011network, schwartz2011network, gueye2012towards, laszka2012game, laszka2013quantifying, szeszler2017security}), the focus of this paper is on the connections between the original game and the theory of $p$-modulus.  The game is illustrated in Figure~\ref{fig:game}. In this game,
	Player \pB (the broadcaster) wants to broadcast a message to all other
	nodes.  This is accomplished by choosing a spanning tree to broadcast on.
	Player \pE (the eavesdropper) can observe the transmission along a single
	link in the link set $\{e_1,e_2,\dots,e_5\}$. Player \pB wins the game if
	the spanning tree avoids \pE's edge, while Player \pE wins if the tree
	includes the edge.  Posed in this way, this is a zero-sum game.

	\begin{figure}
	\centering
	\includegraphics[width=4in]{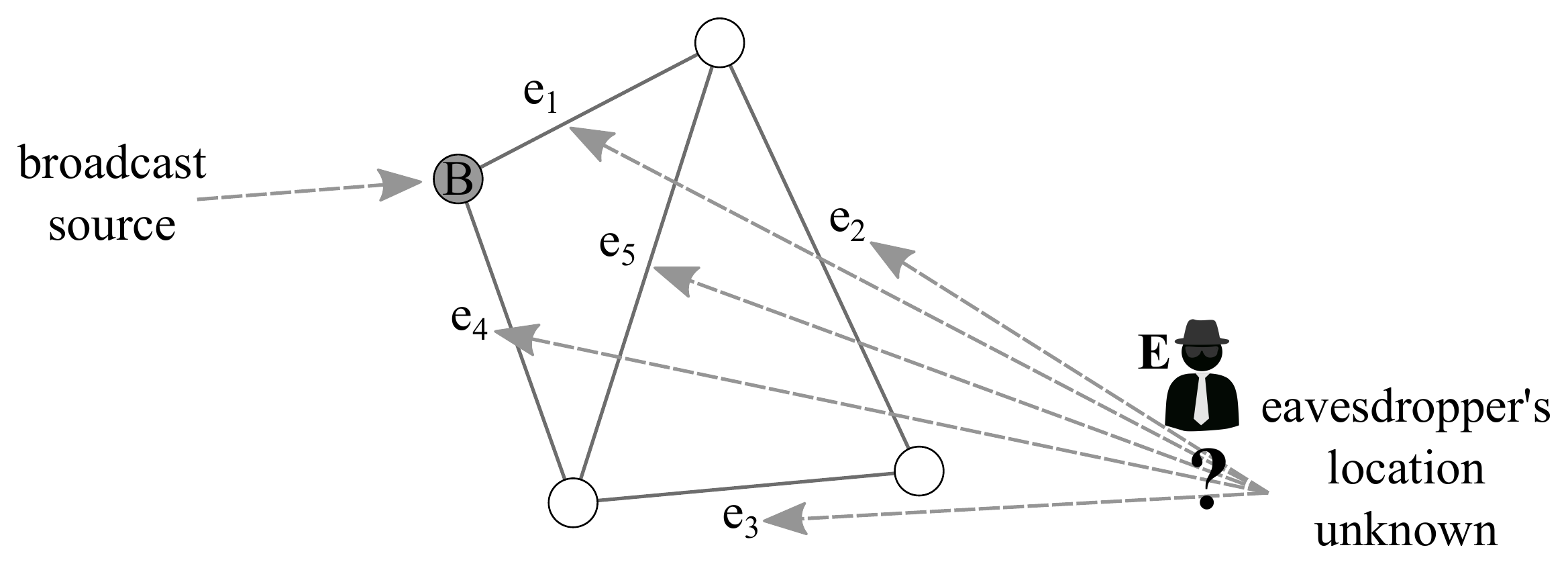}
	\caption{An illustration of the broadcast game.}\label{fig:game}
	\end{figure}

	Figure~\ref{fig:example-spanning-trees} shows four of the eight possible
	spanning trees that \pB can use to broadcast.  It is easy to see that no
	pure-strategy equilibrium exists.  Indeed, if \pB's spanning tree is known
	to \pE, then \pE can readily choose to listen on one of its edges.  Moreover,
	if \pE has chosen an edge, \pB can easily avoid it. On the other
	hand, Nash's existence theorem shows that a mixed-strategy equilibrium
	exists.  Consider the following mixed strategies for Player \pB.

	\begin{figure}
		\centering
		\includegraphics[width=4in]{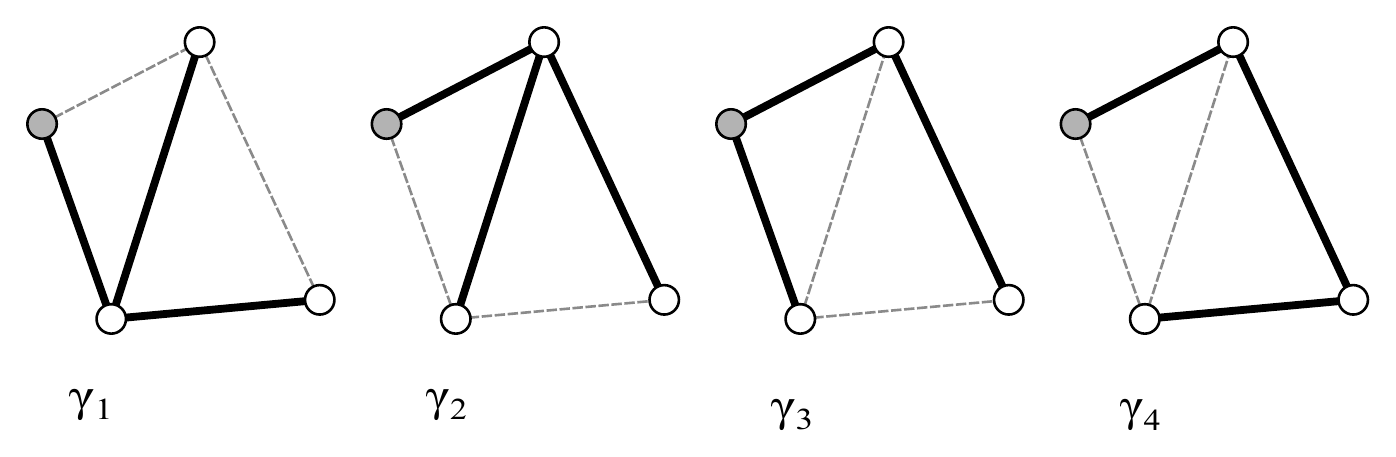}
		\caption{Four spanning trees of the example network in Figure~\ref{fig:game}.}
		\label{fig:example-spanning-trees}
	\end{figure}

	\begin{itemize}
        \item Strategy 1: Use $\gamma_1$ or $\gamma_2$ with equal probability. In this case, \pE can listen on $e_5$ and always win.
        \item Strategy 2: Use $\gamma_1, \gamma_2, \gamma_3, \gamma_4$ with equal probability.  In this case, \pE can choose any strategy that uses only $e_1$ and $e_2$ and win 75\% of the time.
		\item Strategy 3: Use $\gamma_1, \gamma_2, \gamma_3, \gamma_4$ with probabilities 2/5, 1/5, 1/5, 1/5 respectively. In this case, \pE wins 60\% of the time with any strategy.
	\end{itemize}
	In fact, as a consequence of Theorem~\ref{thm:one-shot-solution} (the main theorem of this paper) and Theorem~\ref{thm:homogeneous}, Strategy~3 corresponds to a Nash equilibrium.

	\subsection{Connection to modulus}

	The classical theory of conformal modulus was originally developed in complex analysis several decades ago, see Ahlfors’ comment on p.~81 of~\cite{ahlfors2010conformal}. Intuitively, $p$-modulus provides a method for quantifying the richness of a family of curves: modulus is large for families containing many short curves.  The $p$ in $p$-modulus refers to a parameter $p\in[1,\infty]$.  Modulus tends to favor the ``many curves'' aspect when $p$ is close to 1 and the ``short curves'' aspect as $p$ becomes large. 

	There are several natural analogs to $p$-modulus for families of paths on graphs~\cite{duffin1962extremal,schramm1993square, haissinsky2009empilements,albin2015modulus} and to more general families of objects including spanning trees.  As in the continuum case, modulus provides a quantitative way of describing the richness of a family.  Intuitively, modulus will be large for families containing many objects with small pairwise intersections~\cite{albin2016minimal,albin2018fairest}. This can be made more precise through the minimum expected overlap formulation of modulus (see~\eqref{eq:MEO-opt}), where the goal is to choose random spanning trees that share as few edges as possible on average.  This goal is remarkably similar to that of the broadcaster in the game, which motivates the present work.
	The main contribution of this paper is to establish rigorous connections between the solution of the modulus
	problem on spanning trees and solutions to the broadcast game.
	
	This paper is organized as follows. Section~\ref{sec:modulus} reviews the
	basic definitions and theorems of $p$-modulus along with its probabilistic
	interpretation. Section~\ref{sec:broadcast-games} reviews the broadcast game
	and establishes necessary and sufficient conditions for Nash equilibria.
	Section~\ref{sec:mod-to-solve} establishes the
	connection between modulus and the broadcast game.  Finally, Section~\ref{sec:algorithms-and-examples}
	describes the basic modulus algorithm and demonstrates its use in solving the game.
	
	\subsection{Primary contributions}
	
	The primary contributions of the present paper are the following.
	\begin{itemize}
	    \item Theorem~\ref{thm:necessary-and-sufficient} provides a set of necessary and sufficient conditions for a mixed-strategy solution to the secure broadcast game.
	    \item Theorem~\ref{thm:one-shot-solution} establishes an explicit connection between the $2$-modulus problem (formulated as a \emph{minimum expected overlap} problem) on the spanning trees of a graph and the solution to the game.
	    \item Theorem~\ref{thm:optimal-partition} along with the discussion in Section~\ref{sec:1-vs-2} provides a comparison between the solution method presented in this paper and other methods for solving the game, such as those found in~\cite{gueye2010design}.
	\end{itemize}
	
	\section{Modulus on networks}\label{sec:modulus}
	Let $G = (V,E, \sigma)$ be a finite graph with vertex set $V$, edge set $E$ and edge-weights $\sigma \in \mathbb{R}^{E}_{> 0}$. The graph may be directed or undirected and need not be simple.

	\subsection{Families of objects}

	A \emph{family of objects} on $G$ is a pair $(\Gamma,\mathcal{N})$, where $\Gamma$ is a countable (possibly infinite) index set and $\mathcal{N}\in\mathbb{R}^{\Gamma\times E}_{\ge 0}$ is a non-negative, possibly infinite matrix.  Often, $\Gamma$ alone is referred to as the family and $\mathcal{N}$ is called the \emph{usage matrix} for $\Gamma$.

	Families of objects used in practice often have a simple interpretation on the graph.  For example, it is common to choose $\Gamma\subset 2^E$ and to define the rows of $\mathcal{N}$ as indicator functions.

	\begin{exmp} \label{exmp:characteristic}
	If $\Gamma$ is the set of spanning trees of $G$ (thought of as subsets of $E$), we can define
    \begin{equation*}
      \mathcal{N}(\gamma,e) := \mathbbm{1}_\gamma(e)
   	\end{equation*}
    for $\gamma\in\Gamma$ and $e\in E$.
	\end{exmp}

	This is the family we shall use in the present paper, however the theory of modulus applies to more general families of objects.

	\begin{exmp}
	If $a$ and $b$ are distinct vertices and if $\Gamma$ is the family of walks in $G$ beginning at $a$ and ending at $b$, there are many different options for defining $\mathcal{N}$.  For example, we could choose any of the following definitions.
	\begin{itemize}
		\item $\mathcal{N}(\gamma,e) = \mathbbm{1}_\gamma(e)$,
		\item $\mathcal{N}(\gamma,e) = $ the number of times $\gamma$ crosses edge $e$,
		\item $\mathcal{N}(\gamma,e) = \ell(\gamma)^{-1}\mathbbm{1}_\gamma(e)$ where $\ell(\gamma)$ is the length of $\gamma$.
	\end{itemize}
	\end{exmp}

	Since our goal in this paper is primarily concerned with the modulus of spanning trees, we will assume in the following discussion that the family $\Gamma$ is finite.  We further assume that $\Gamma$ is \emph{nontrivial} in the sense that $\Gamma\ne\varnothing$ and each object $\gamma\in\Gamma$ has positive usage on at least one edge. (So $\mathcal{N}$ has at least one row, and no row of $\mathcal{N}$ is identically zero.)  While $\mathcal{N}$ can be thought of as a function $\mathcal{N}:\Gamma\times E\to\mathbb{R}_{\ge 0}$, it is often notationally convenient to represent $\mathcal{N}$ as a matrix with rows indexed by $\Gamma$ and columns indexed by $E$. So the $(\gamma,e)$ entry keeps a record of the usage of the edge $e$ by the object $\gamma$.

	\subsection{Densities and admissible densities}

	A \emph{density} on $G$ is a vector $\rho\in \mathbb{R}^E_{\geq 0}$. Choosing a density can be thought of as assigning a cost $\rho(e)$ for the use of each edge $e$. Each density, $\rho$, induces a \emph{$\rho$-length} or \emph{total usage cost} on the family of objects $\Gamma$.  The $\rho$-length of $\gamma\in\Gamma$ is defined as
	\begin{equation}\label{eq:rho-length}
		\ell_{\rho}(\gamma) = \sum_{e \in E} \mathcal{N}(\gamma,e)\rho(e) = (\mathcal{N}\rho)(\gamma).
	\end{equation}

	The modulus is defined through the minimization of an energy (defined shortly) over a set of admissible densities. We say that a density $\rho$ is \emph{admissible for $\Gamma$} if,
	\begin{equation*}
		\ell_{\rho}(\gamma) \geq 1 \hspace*{3mm} \forall \gamma \in \Gamma,\qquad
		\text{or, equivalently,}\qquad
		\mathcal{N}\rho \geq \textbf{1},
	\end{equation*}
	where $\textbf{1}$ is the vector of all ones in $\mathbb{R}^{\Gamma}$ and the inequality is understood elementwise. Using the latter notation, we define the set of all admissible densities as
	\begin{equation*}
		\Adm \Gamma = \{\rho \in \mathbb{R}^{E}_{\geq 0}: \mathcal{N}\rho \geq \textbf{1} \}.
	\end{equation*}
	An alternative notation that is sometimes useful is to define
	\begin{equation}\label{eq:rho-length-Gamma}
	\ell_\rho(\Gamma) := \min_{\gamma\in\Gamma}\ell_\rho(\gamma).
	\end{equation}
	So another way of describing the admissible set is,
	\begin{equation*}
	\Adm \Gamma = \{\rho\in\mathbb{R}^E_{\ge 0} : \ell_\rho(\Gamma)\ge 1\}.
	\end{equation*}

	\subsection{Energy and modulus}

	Given an exponent $p\geq1$ we define the \emph{$p$-energy} of a density $\rho$ as
	\begin{equation*}
	\mathcal{E}_{p,\sigma}(\rho) = \sum\limits_{e \in E} \sigma(e) \rho(e)^p.
	\end{equation*}
	For $p = \infty$, we also define $\infty$-energy as
	\begin{equation*}
	\mathcal{E}_{\infty,\sigma}(\rho) = \lim_{p \rightarrow \infty}(\mathcal{E}_{p,\sigma^p}(\rho))^{\frac{1}{p}} = \max\limits_{e \in E} \sigma(e)\rho(e).
	\end{equation*}
	\begin{definition}
		Given a graph $G = (V, E, \sigma)$, a family of objects $\Gamma$ with usage matrix $\mathcal{N} \in \mathbb{R}^{\Gamma\times E}$,
		and an exponent $1 \leq p \leq \infty$, the $p$-modulus of $\Gamma$ is,
		\begin{equation*}
			\Mod_{p,\sigma}(\Gamma) = \inf\limits_{\rho \in \Adm \Gamma} \mathcal{E}_{p,\sigma}(\rho).
		\end{equation*}
	\end{definition}
	
	In standard optimization notation, $p$-modulus is the value of the problem
	\begin{equation} \label{eq:cvxp}
	\begin{split}
	\text{minimize} &\hspace{8mm} \mathcal{E}_{p,\sigma}(\rho) \\
	\text{subject to} &\hspace{8mm} \mathcal{N}\rho \geq \textbf{1} \hspace{3mm} \text{and} \hspace{3mm} \rho \ge \textbf{0}
	\end{split}
	\end{equation}
	where each object $\gamma \in \Gamma$ determines an inequality constraint.  Written in the form of~\eqref{eq:cvxp}, it is evident that the modulus problem is an ordinary convex optimization problem; geometrically, it can be described as the problem of finding the distance (in a weighted $p$-norm) between the closed set $\Adm\Gamma$ and the origin. Thus, a minimizer always exists, and uniqueness holds when $ 1 < p < \infty$ due to the strict convexity of the objective function (see \cite{albin2016minimal}).

	Since this paper is focused entirely on the unweighted ($\sigma\equiv 1$), and $p=2$ case, we will typically drop the $\sigma$ from the subscripts above and instead write $\mathcal{E}_2$ and $\Mod_2$ for the energy and modulus respectively. As described in~\cite{albin2018fairest}, this assumption is not very restrictive since the case of weighted graphs can always be approximated by unweighted multigraphs.

	\subsection{Probabilistic interpretation}\label{sec:probabilistic-interpretation}

	The relationship between modulus and secure broadcast games arises from the probabilistic interpretation of modulus developed in~\cite{albin2016minimal}. In this interpretation, we consider an object $\underline{\gamma}$ chosen at random according to a probability mass function (pmf) $\mu$.  (The underline in the notation $\underline{\gamma}$ is used to distinguish the random object from its possible values.)  In other words, for each $\gamma\in\Gamma$, $\mu(\gamma)$ defines the probability that $\underline{\gamma}=\gamma$ or, in simpler notation, $\mu(\gamma)=\mathbb{P}_\mu(\underline{\gamma}=\gamma)$.  Here we use the subscript notation $\mathbb{P}_\mu$ to specify exactly which pmf is being used.  We will also represent the relationship between the random object $\underline{\gamma}$ and its pmf $\mu$ by the notation $\underline{\gamma}\sim\mu$.  The set of all pmfs on $\Gamma$ will be represented with the notation $\mathcal{P}(\Gamma) := \big\{\mu \in \mathbb{R}^{\Gamma}_{\geq 0}: \mu^T\textbf{1} = 1 \big\}$.

	If $e\in E$, $\mu\in\mathcal{P}(\Gamma)$ and $\underline{\gamma}\sim\mu$, then $\mathcal{N}(\underline{\gamma},e)$ is a real-valued random variable and we denote its expected value with respect to $\mu$ as $\mathbb{E}_{\mu}[\mathcal{N}(\underline{\gamma},e)]$. That is,
	\begin{equation}\label{eq:expected-usage}
	\mathbb{E}_{\mu}[\mathcal{N}(\underline{\gamma},e)] = \sum\limits_{\gamma \in \Gamma}\mathcal{N}(\gamma, e)\mu(\gamma) = (\mathcal{N}^T\mu)(e).
	\end{equation}
	Thus, $(\mathcal{N}^T\mu)(e)$ represents the \emph{expected usage} of edge $e$ by the random object $\underline{\gamma}\sim\mu$.  We will frequently use the variable $\eta$ to represent the vector $\mathcal{N}^T\mu\in\mathbb{R}^E$. 

	With this notation, we can now summarize the relevant result from~\cite{albin2016minimal}.  (This theorem is also a specialization of \cite[Theorem~2.8]{albin2018blocking} to the $p=2$ case.)
	\begin{theorem}
		Let $G = (V, E, \sigma)$ be a graph and let $\Gamma$ be a nontrivial finite family of objects on $G$ with usage matrix $\mathcal{N}$. Then we have,
		\begin{equation}
		\Mod_{2,\sigma}(\Gamma)^{-1} = \min\limits_{\mu \in \mathcal{P}(\Gamma)} \sum\limits_{e\in E}\sigma(e)^{-1}\mathbb{E}_{\mu}[\mathcal{N}(\underline{\gamma},e)]^2. \label{mod}
		\end{equation}
		The minimum is always attained and $\mu\in\mathcal{P}(\Gamma)$ is optimal if and only if
		\begin{equation}\label{eq:expected-usage-vs-rho}
		\mathbb{E}_{\mu}[\mathcal{N}(\underline{\gamma},e)] = \frac{\sigma(e)\rho^*(e)}{\Mod_{2,\sigma}(\Gamma)},
		\end{equation}
		where $\rho^*$ is the unique extremal density for $\Mod_{2,\sigma}(\Gamma)$.
	\end{theorem}
	
	 This probabilistic interpretation has a simple meaning under the assumptions that $G$ is unweighted (i.e., $\sigma \equiv 1$), $\Gamma$ is a collection of subsets of $E$, and $\mathcal{N}$ is a matrix of 0's and 1's (i.e., $\mathcal{N}(\gamma,e)$ is defined through indicator functions as in Example~\ref{exmp:characteristic}). Since $\mathcal{N}$ is a $(0,1)$-matrix, $\mathcal{N}(\underline{\gamma},e$) is an indicator random variable for the event $e\in\underline{\gamma}$.  In this case, the expected usage 
	\begin{equation}\label{eq:eta-as-expectation}
	\eta(e):=\mathbb{E}_{\mu}[\mathcal{N}(\underline{\gamma},e)]
	= (\mathcal{N}^T\mu)(e)
	\end{equation}
	in equations~\eqref{eq:expected-usage} and~\eqref{mod} is actually a probability of inclusion
	\begin{equation}\label{eq:eta-as-prob}
	\eta(e) = \mathbb{P}_{\mu}(e\in\underline{\gamma}).
	\end{equation}
	Moreover, the summation in the right-hand side of~\eqref{mod} can be rewritten as
	\begin{equation}\label{eq:eta-as-EO}
	\begin{split}
	\sum_{e\in E}\eta(e)^2 &=
	\sum_{e\in E}\left(\sum_{\gamma\in\Gamma}
	\mathcal{N}(\gamma,e)\mu(\gamma)\right)
	\left(\sum_{\gamma'\in\Gamma}\mathcal{N}(\gamma',e)\mu(\gamma')\right)\\
	&= \sum_{\gamma\in\Gamma}\sum_{\gamma'\in\Gamma}
	\left(\sum_{e\in E}\mathcal{N}(\gamma,e)\mathcal{N}(\gamma',e)\right)
	\mu(\gamma)\mu(\gamma')\\
	&= \sum_{\gamma\in\Gamma}\sum_{\gamma'\in\Gamma}
	|\gamma\cap\gamma'|
	\mu(\gamma)\mu(\gamma') =: 
	\mathbb{E}_{\mu}|\underline{\gamma}\cap\underline{\gamma'}|.
	\end{split}
	\end{equation}
	Here $\underline{\gamma},\underline{\gamma'}\sim\mu$ are two independent $\Gamma$-valued random variables. The quantity $|\underline{\gamma}\cap\underline{\gamma'}|$ is the intersection (overlap) of these two random sets and is, therefore, an integer-valued random variable with expectation $\mathbb{E}_{\mu}|\underline{\gamma}\cap\underline{\gamma'}|$.  This latter quantity is called the \emph{expected overlap} of $\underline{\gamma}$ and $\underline{\gamma'}$.

	Using~\eqref{eq:eta-as-EO}, the probabilistic interpretation~\eqref{mod} can be expressed as follows. 
	\begin{equation}\label{eq:MEO}
	\Mod_{2}(\Gamma)^{-1} = \min\limits_{\mu \in \mathcal{P}(\Gamma)} \mathbb{E}_{\mu}|\underline{\gamma}\cap\underline{\gamma'}|.
	\end{equation}

	 Thus, computing $2$-modulus in this case is equivalent to finding a pmf that minimizes the expected overlap of two iid $\Gamma$-valued random variables. The right-hand side of \eqref{eq:MEO} is called the \emph{minimum expected overlap} (MEO) problem:
	\begin{equation}\label{eq:MEO-opt}
	\begin{split}
	\text{minimize} &\hspace{8mm} \mathbb{E}_{\mu}|\underline{\gamma}\cap\underline{\gamma'}| \\
	\text{subject to} &\hspace{8mm} \mu \in \mathcal{P}(\Gamma).
	\end{split}
	\end{equation}

	When $\mu^*$ is optimal for the MEO problem~\eqref{eq:MEO-opt}, we define $\eta^*=\mathcal{N}^T\mu^*$. By~\eqref{eq:expected-usage-vs-rho},
	\begin{equation}
	\eta^*(e) = \mathbb{P}_{\mu^*}(e\in\underline{\gamma}) = \mathbb{E}_{\mu^*}[\mathcal{N}(\underline{\gamma},e)] 
	= \frac{\rho^*(e)}{\Mod_{2}(\Gamma)} \label{eta*}.
	\end{equation}
	Since $\rho^*$ is unique, so is $\eta^*$. In general, however, the pmf $\mu^*$ is non-unique.  

	\begin{exmp}
	Consider the paw graph shown in Figure~\ref{fig:paw-trees}, with edges enumerated as indicated. This graph has 3 spanning trees: $\Gamma = \{\gamma_1, \gamma_2, \gamma_3\}$.	The pairwise overlaps among these trees satisfy
	\begin{equation*}
	|\gamma_i\cap\gamma_j| =
	\begin{cases}
	3 & \text{if  } i = j,\\
	2 & \text{otherwise}.
	\end{cases}
	\end{equation*}
	Let the density $\rho$ take the value 3/7 on edge 1 and 2/7 on other edges. This density is admissible.  Thus, its energy $\rho^T\rho = 3/7$ forms an upper bound for $\Mod_{2}(\Gamma)$. 

	On the other hand, let $\mu\in\mathcal{P}(\Gamma)$ be the uniform distribution. Then the expected overlap is,
	\begin{equation*}
	\mathbb{E}_{\mu}|\underline{\gamma}\cap\underline{\gamma'}| = 
	3\mathbb{P}_\mu(\underline{\gamma}=\underline{\gamma'})
	+ 2\mathbb{P}_\mu(\underline{\gamma}\ne\underline{\gamma'})
	=
	3\cdot\frac{1}{3} + 2\cdot\frac{2}{3} = \frac{7}{3}, 
	\end{equation*}
	which forms an upper bound for $\Mod_{2}(\Gamma)^{-1}$ by~\eqref{eq:MEO}. That is,
	\begin{equation*}
	\frac{3}{7} = \frac{1}{\mathbb{E}_{\mu}|\underline{\gamma}\cap\underline{\gamma'}|} \leq \Mod_{2}(\Gamma) \leq \rho^T\rho = \frac{3}{7}.		
	\end{equation*}
	Therefore,
	\begin{equation*}
	\Mod_{2}(\Gamma) = 3/7,
	\end{equation*}
	and $\rho$ and $\mu$ are optimal for their corresponding minimization problems.
	\end{exmp}

	\begin{figure}
	\begin{center}
		\includegraphics[width=6in]{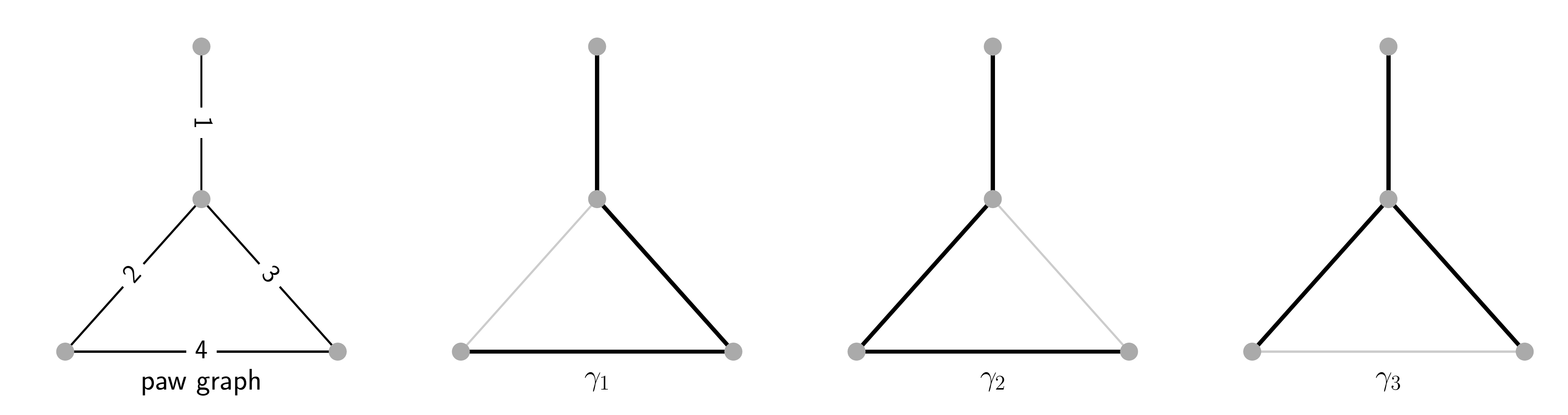}
		\caption{\text{Spanning trees of the paw graph.}}
	\label{fig:paw-trees}
	\end{center}
	\end{figure}
	\begin{exmp}
	Consider the graph in Figure~\ref{fig:game}. This graph has 8 spanning trees, as shown in Figure~\ref{fig:diamond-trees}.  To derive an upper bound on $\Mod_2(\Gamma)$, observe that $\rho\equiv 1/3$ is an admissible density.  Therefore,
	\begin{equation}\label{eq:diamond-ubound}
	\Mod_2(\Gamma) \le \mathcal{E}_2(\rho) = \frac{5}{9}.
	\end{equation}

	Now, consider the pmf $\mu\in\mathcal{P}(\Gamma)$ defined as
	\begin{equation*}
	\mu(\gamma) =
	\begin{cases}
	\tfrac{3}{20} & \text{if } \gamma \text{ contains the diagonal},\\
	\tfrac{2}{20} & \text{otherwise}.
	\end{cases}
	\end{equation*}
	Let $e_d$ be the diagonal edge.	Since there are four spanning trees containing $e_d$,
	\begin{equation*}
	\eta(e_d) = 4\cdot\frac{3}{20} = \frac{3}{5}.
	\end{equation*}
	For all other edges, $e\ne e_d$
	\begin{equation*}
	\begin{split}
	\eta(e) & = \mathbb{P}_{\mu}(e\in\underline{\gamma}) = \mathbb{P}_{\mu}(e\in\underline{\gamma}|e_d\in\underline{\gamma})
	\mathbb{P}_{\mu}(e_d\in\underline{\gamma})
	+ \mathbb{P}_{\mu}(e\in\underline{\gamma}|e_d\notin\underline{\gamma})
	\mathbb{P}_{\mu}(e_d\notin\underline{\gamma})\\
	&= \frac{1}{2}\left(4\cdot\frac{3}{20}\right)
	+ \frac{3}{4}\left(4\cdot\frac{2}{20}\right) = \frac{3}{5}.
	\end{split}	
	\end{equation*}
	With this pmf, all edges are equally likely to occur in $\underline{\gamma}$: $\eta(e)=3/5$ on all edges. Therefore the expected overlap is, by~\eqref{eq:eta-as-EO},
	\begin{equation*}
	\mathbb{E}_{\mu}|\underline{\gamma}\cap\underline{\gamma'}| = 
	\sum_{e\in E}\eta(e)^2 = 
	5\bigg(\frac{3}{5}\bigg)^2 = \frac{9}{5} \geq \Mod_{2}(\Gamma)^{-1},
	\end{equation*}
	which forms a lower bound for $\Mod_{2}(\Gamma)$. This together with the upper bound~\eqref{eq:diamond-ubound} shows that both bounds are attained.

	Although the optimal vectors $\rho^*$ and $\eta^*$ are unique, the optimal pmf $\mu^*$ need not be.  In Section~\ref{sec:game-intro}, we described another pmf, namely
	\begin{equation*}
	\mu = \frac{2}{5}\delta_{\gamma_1} + \frac{1}{5}\delta_{\gamma_5} + \frac{1}{5}\delta_{\gamma_6} + \frac{1}{5}\delta_{\gamma_8}.
	\end{equation*}
	It is straightforward to verify that this pmf also has the property that $\mathbb{P}_\mu(e\in\underline{\gamma})=3/5$ for each edge $e$ and is therefore optimal.
	
	The uniform pmf $\mu\equiv 1/8$, on the other hand, is not optimal in this case, since $\eta(e_d)=1/2$ while for all other edges $e\ne e_d$,
	$\eta(e) = 5/8$.  For the uniform pmf, then, the expected overlap is
	\begin{equation*}
	\mathbb{E}_{\mu}|\underline{\gamma}\cap\underline{\gamma'}| 
	= \sum_{e\in E}\eta(e)^2 = \left(\frac{1}{2}\right)^2
	+ 4\left(\frac{5}{8}\right)^2
	= \frac{29}{16} > \frac{9}{5}.
	\end{equation*}
	\end{exmp}

	\begin{figure}
	\begin{center}
		\includegraphics[width=6in]{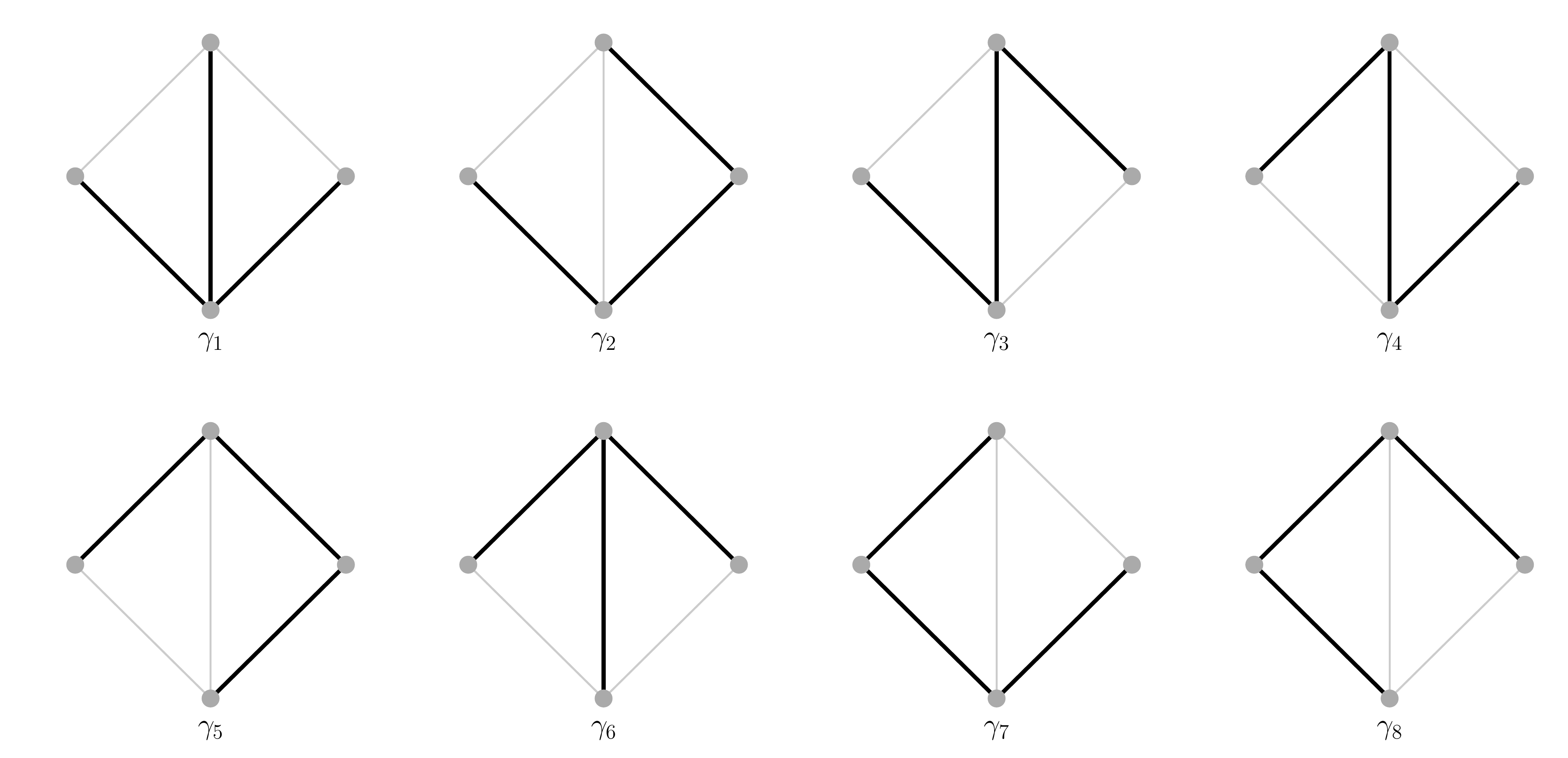}
		\caption{\text{Spanning trees of the diamond graph.}}
		\label{fig:diamond-trees}
	\end{center}
	\end{figure}

	\section{A secure broadcast game}\label{sec:broadcast-games}

	When addressing secure broadcasting (or network security in general) a
	game theoretic approach is natural due to the rich set of mathematical
	tools that this approach provides. In this
	section, we consider the solution of the secure broadcast game described in Section~\ref{sec:introduction}.  We show that 
	the solution to the spanning tree modulus
	problem provides a Nash equilibrium for the game.

	To more carefully define the game, let $G=(V,E)$ be a simple, connected graph representing a communications network and let $\Gamma$ be the family of spanning trees on $G$.  The game between \pB (the broadcaster) and
	\pE (the eavesdropper) proceeds as follows.
	Player \pB broadcasts a message by sending it along a spanning tree
	$\gamma\in\Gamma$.  Player \pE eavesdrops by choosing a single edge $e\in E$ on which to listen.  \pE wins if $e\in\gamma$.  Otherwise, \pB wins.  If we consider this a zero-sum game, we may assume without loss of generality that the payoff matrix $P$ is equal to $\mathcal{N}$.  That is, \pE wishes to maximize $\mathcal{N}(\gamma,e)$ while \pB wishes to minimize it.  Written in this way, it is evident that a pure-strategy solution to the game generally does not exist.  In fact, the following theorem gives necessary and sufficient conditions on the existence of a pure-strategy equilibrium.

	\begin{theorem}\label{thm:pure-strategy}
	A pure-strategy equilibrium to the secure broadcast game exists if and only if $G$ is 1-edge-connected (that is, if there exists an edge $e\in E$ whose removal would disconnect the graph).
	\end{theorem}
	\begin{proof}
	First, suppose that such an edge $e$ exists.  Necessarily, this edge exists in every spanning tree and, thus, $\mathcal{N}(\gamma,e)=1$ for all $\gamma\in\Gamma$.  \pE wins by selecting this edge.  Since \pE can force a win, this provides a pure-strategy equilibrium to the game.

	Now, suppose that $G$ is at least 2-edge-connected.  Assume that \pB plays first, by choosing a tree $\gamma\in\Gamma$.  Then \pE wins by choosing any $e\in\gamma$.  On the other hand, assume that \pE plays first by choosing an edge $e\in E$.  Let $G'$ be the subgraph of $G$ obtained by removing $e$.  By assumption, $G'$ is connected, so there exists a spanning tree $\gamma\in\Gamma$ that is also a spanning tree of $G'$.  \pB wins by selecting this tree.  Since neither player can force a win in this case, there is no pure-strategy equilibrium.
	\end{proof}

	In light of Theorem~\ref{thm:pure-strategy}, a Nash equilibrium for 
	the secure broadcast game will generally involve mixed strategies.  Here we recast the game using probability notation in order to make the connection to modulus more apparent. More precisely, we suppose that \pB chooses a pmf
	$u\in\mathcal{P}(\Gamma)$ while \pE chooses a pmf $v\in\mathcal{P}(E)$.
	This is equivalent to \pB choosing a random tree $\underline{\gamma}\sim u$ while \pE chooses a random edge $\underline{e}\sim v$.  The expected payoff for a particular choice of strategies, then, is given by
	\begin{equation}\label{eq:exp-payoff-one-shot}
	\mathbb{E}_{u,v}[\mathcal{N}(\underline{\gamma},\underline{e})]
	:=  \sum_{\gamma\in\Gamma}\sum_{e\in E}\mathcal{N}(\gamma,e)u(\gamma)v(e).
	\end{equation}
	This quantity has a number of useful interpretations.  For example, if we define $\eta:=\mathcal{N}^T u=\mathbb{P}_u(\cdot\in\underline{\gamma})$ as in~\eqref{eq:eta-as-prob}, then
	\begin{equation}\label{eq:payoff-as-eta-dot-v}
	\mathbb{E}_{u,v}[\mathcal{N}(\underline{\gamma},\underline{e})]
	= \sum_{e\in E}\eta(e)v(e) = \mathbb{E}_v[\eta(\underline{e})].
	\end{equation}
	That is, the expected payoff is equal to the expected value of $\eta$ on an edge randomly selected according to the distribution $v$.
	On the other hand, if we us the definition of $\ell_v(\gamma)$ in~\eqref{eq:rho-length}, then
	\begin{equation}\label{eq:payoff-as-v-length}
	\mathbb{E}_{u,v}[\mathcal{N}(\underline{\gamma},\underline{e})]
	= \sum_{\gamma\in \Gamma}\ell_v(\gamma)u(\gamma)
	= \mathbb{E}_u[\ell_v(\underline{\gamma})].
	\end{equation}
	With this interpretation, the expected payoff is the expected $v$-length of a random spanning tree chosen according to the distribution $u$.  Equations~\eqref{eq:payoff-as-eta-dot-v} and~\eqref{eq:payoff-as-v-length} provide a useful way of looking at necessary and sufficient conditions for solving the game.	We begin by reviewing the relevant concepts in the present context.

	First, recall that a pair of mixed strategies $(u^*,v^*)\in\mathcal{P}(\Gamma)\times\mathcal{P}(E)$ solves the game (in the sense of Nash equilibrium~\cite{Nash1950}) if the following condition holds.
	\begin{equation}\label{eq:saddle-point}
	\mathbb{E}_{u^*,v}[\mathcal{N}(\underline{\gamma},\underline{e})]
	\le
	\mathbb{E}_{u^*,v^*}[\mathcal{N}(\underline{\gamma},\underline{e})]
	\le
	\mathbb{E}_{u,v^*}[\mathcal{N}(\underline{\gamma},\underline{e})]
	\qquad\forall u\in\mathcal{P}(\Gamma)\;\forall v\in\mathcal{P}(E).	
	\end{equation}
	Alternatively, we can express~\eqref{eq:saddle-point} as
	\begin{equation}\label{eq:saddle-point-2}
	\max_{v\in\mathcal{P}(E)}
	\mathbb{E}_{u^*,v}[\mathcal{N}(\underline{\gamma},\underline{e})]
	= \mathbb{E}_{u^*,v^*}[\mathcal{N}(\underline{\gamma},\underline{e})]
	=
	\min_{u\in\mathcal{P}(\Gamma)}
	\mathbb{E}_{u,v^*}[\mathcal{N}(\underline{\gamma},\underline{e})].
	\end{equation}
	This form is particularly enlightening when paired with the minimax theorem, which states that for a general real-valued function $f$ on a product space $X\times Y$,
	\begin{equation*}
	    \sup_{y\in Y}\inf_{x\in X}f(x,y)\le
	    \inf_{x\in X}\sup_{y\in Y}f(x,y).
	\end{equation*}
	When applied to the problem at hand, the minimax theorem becomes
	\begin{equation}\label{eq:minimax}
	\max_{v\in\mathcal{P}(E)}
	\min_{u\in\mathcal{P}(\Gamma)}
	\mathbb{E}_{u,v}[\mathcal{N}(\underline{\gamma},\underline{e})]
	\le
	\min_{u\in\mathcal{P}(\Gamma)}
	\max_{v\in\mathcal{P}(E)}
	\mathbb{E}_{u,v}[\mathcal{N}(\underline{\gamma},\underline{e})].
	\end{equation}
	Nash's existence theorem~\cite{Nash1950} guarantees a mixed-strategy solution to the game and, therefore, that the inequality in~\eqref{eq:minimax} is satisfied as equality.  These observations lead to the following theorem.

	\begin{theorem}\label{thm:necessary-and-sufficient}
	Let $u\in\mathcal{P}(\Gamma)$ and $v\in\mathcal{P}(E)$ and define
	\begin{equation*}
	\eta=\mathcal{N}^Tu\qquad\text{and}\qquad
	\ell_v(\Gamma) = \min_{\gamma\in\Gamma}\ell_v(\gamma).
	\end{equation*}
	Then $(u,v)$ solves the game if and only if
	\begin{equation}\label{eq:necessary-and-sufficient}
	\ell_v(\Gamma) \ge \|\eta\|_\infty.
	\end{equation}

	Moreover, if $(u,v)$ solves the game, then
	\begin{equation}\label{eq:necessary-u}
	\supp u\subseteq\{\gamma\in\Gamma:\ell_v(\gamma)=\ell_v(\Gamma)\}
	\end{equation}
	and
	\begin{equation}\label{eq:necessary-v}
	\supp v\subseteq\{e\in E:\eta(e)=\|\eta\|_\infty\}.
	\end{equation}
	\end{theorem}

	\begin{proof}
	From~\eqref{eq:payoff-as-eta-dot-v}, we have
	\begin{equation*}
	\max_{v\in\mathcal{P}(E)}
	\mathbb{E}_{u,v}[\mathcal{N}(\underline{\gamma},\underline{e})]
	= 
	\max_{v\in\mathcal{P}(E)}
	\mathbb{E}_v[\eta(\underline{e})]
	= \|\eta\|_\infty.
	\end{equation*}
	Similarly, by~\eqref{eq:payoff-as-v-length},
	\begin{equation*}
	\min_{u\in\mathcal{P}(\Gamma)}
	\mathbb{E}_{u,v}[\mathcal{N}(\underline{\gamma},\underline{e})]
	=
	\min_{u\in\mathcal{P}(\Gamma)}
	\mathbb{E}_u[\ell_v(\underline{\gamma})]
	=
	\ell_v(\Gamma).
	\end{equation*}
	Substituting these two expressions into~\eqref{eq:saddle-point-2} shows that an equilibrium is characterized by the equality
	\begin{equation*}
	\|\eta\|_\infty = \ell_v(\Gamma).
	\end{equation*}
	Moreover, substituting into the minimax theorem~\eqref{eq:minimax} shows that
	\begin{equation*}
		\max_{v\in\mathcal{P}(E)}\ell_v(\Gamma) \le 
		\min_{u\in\mathcal{P}(\Gamma)}\|\eta\|_\infty\qquad.
	\end{equation*}
	Thus,~\eqref{eq:necessary-and-sufficient} is necessary and sufficient to characterize an equilibrium.

	The additional necessary conditions~\eqref{eq:necessary-u} and~\eqref{eq:necessary-v} are a consequence of~\eqref{eq:saddle-point-2}.  Let $(u,v)$ be an equilibrium point and let $\eta=\mathcal{N}^Tu$.  Then,~\eqref{eq:saddle-point-2} implies that
	\begin{equation}\label{eq:eta-max-as-game-value}
	\|\eta\|_\infty = 
	\mathbb{E}_{u,v}[\mathcal{N}(\underline{\gamma},\underline{e})]
	=
	\mathbb{E}_v[\eta(\underline{e})].
	\end{equation}
	This equality can only hold if $v$ is supported on edges where $\eta$ attains its maximum, establishing~\eqref{eq:necessary-v}.  A similar argument establishes~\eqref{eq:necessary-u}.
	\end{proof}

	\begin{remark}
	Theorem~\ref{thm:pure-strategy} can also be seen as a corollary of Theorem~\ref{thm:necessary-and-sufficient}.  Indeed, if $\gamma\in\Gamma$ and $e\in E$ then the pure strategies $u=\delta_\gamma\in\mathcal{P}(\Gamma)$ and $v=\delta_e\in\mathcal{P}(E)$ have the properties that
	\begin{equation*}
	\|\eta\|_\infty = 1\qquad\text{and}\qquad \ell_v(\gamma') = \mathcal{N}(\gamma',e)\quad\forall\gamma'\in\Gamma.
	\end{equation*}
	Thus, by Theorem~\ref{thm:necessary-and-sufficient}, the pair $(u,v)$ solves the game if and only if $e\in\gamma'$ for every $\gamma'\in\Gamma$, that is, if and only if $e$ is a ``bridge'' in the graph.
	\end{remark}

	\section{Using modulus to solve the game}\label{sec:mod-to-solve}

	Now, we show that the solution to the spanning tree modulus problem provides a solution to the game.  The connection is made through the concept of feasible partitions, as defined in~\cite{Chopra1989}.

	\begin{definition}
	A \emph{feasible partition} $Q$ of a graph $G=(V,E)$ is a partition of the
	vertex set $V$ into $k_Q>1$ pieces: $Q=\{V_1,V_2,\ldots,V_{k_Q}\}$ such that
	the subgraph of $G$ induced by $V_i$ is connected for each
	$i=1,2,\ldots,k_Q$.  The edge set $E_Q\subseteq E$ of
	$Q$ is the set of all edges of $E$ that connect distinct pieces of the
	partition.  The \emph{weight} $w(Q)$ of $Q$ is defined as $w(Q)=|E_Q|/(k_Q-1)$.  In the case $k_Q=2$, the definition of a feasible partition coincides with the definition of a global graph cut and $w(Q)=|E_Q|$ is the usual definition of the weight of such a cut.  We denote by $\mathcal{F}(G)$ the set of all feasible partitions of $G$.
	\end{definition}

	An important relationship between spanning trees and feasible partitions arises from the fact that, for a given feasible partition $Q$ and spanning tree $\gamma$, the spanning tree must connect all pieces of the partition, yielding the inequality (see~\cite[(2.2)]{Chopra1989})
	\begin{equation}\label{eq:basic-fp-bound}
	|\gamma\cap E_Q| \ge k_Q-1.
	\end{equation}

	The relationship between modulus and the secure broadcast game is provided by the following theorem, proved in~\cite{albin2018fairest}.
	\begin{theorem}\label{thm:optimal-partition}
		Let $\Gamma$ be the family of spanning trees on the graph $G=(V,E)$ and let $\mu^* \in \mathcal{P}(\Gamma)$ be optimal for the MEO problem~\eqref{eq:MEO-opt}, with expected edge usage $\eta^* = \mathcal{N}^T\mu^*$.  Define
		\begin{equation*}
		E_{Q^*} := \{e\in E:\eta^*(e) = \|\eta^*\|_\infty\}.
		\end{equation*}
		Then $E_{Q^*}$ is the edge set for a feasible partition $Q^*$ of $V$ into $k_{Q^*}>1$ pieces and has the following three properties.
		\begin{eqnarray}
			|\gamma \cap E_{Q^*}| &=& k_{Q^*} - 1 \hspace*{3mm} \forall \gamma \in \supp\ \mu^*,\\
			\eta^*(e) = \|\eta^*\|_\infty &=& w(Q^*)^{-1} \hspace*{3mm} \forall e \in E_{Q^*},
			\label{eq:max-eta-star}\\
			w(Q^*) &=& \min_{Q\in\mathcal{F}(G)}w(Q).
		\end{eqnarray}
	\end{theorem}

	This provides the solution to the secure broadcast game as follows.

	\begin{theorem}\label{thm:one-shot-solution}
	Using the definitions from Theorem~\ref{thm:optimal-partition}, let $u^*=\mu^*\in\mathcal{P}(\Gamma)$ and let $v^*\in\mathcal{P}(E)$ be the uniform distribution on $E_{Q^*}$. Then $(u^*,v^*)$ is a Nash equilibrium solution to the game.  Moreover,
	\begin{equation}\label{eq:game-value}
	\mathbb{E}_{u^*,v^*}[\mathcal{N}(\underline{\gamma},\underline{e})] =
	\|\eta^*\|_{\infty} = w(Q^*)^{-1}.
	\end{equation}
	\end{theorem}

	\begin{proof}
	In order to verify that the given mixed strategies solve the game, we use Theorem~\ref{thm:necessary-and-sufficient}, which shows that it is sufficient to establish that~\eqref{eq:necessary-and-sufficient} holds.  Equation~\eqref{eq:max-eta-star} establishes the second equality in~\eqref{eq:game-value}.  Let $\gamma\in\Gamma$ be any spanning tree.  Then, since $Q^*$ is a feasible partition,~\eqref{eq:basic-fp-bound} shows that
	\begin{equation*}
	\ell_{v^*}(\gamma) = \sum_{e\in E}\mathcal{N}(\gamma,e)v^*(e)
	= \frac{1}{|E_{Q^*}|}\sum_{e\in E_{Q^*}}\mathcal{N}(\gamma,e)
	=\frac{|\gamma\cap E_{Q^*}|}{|E_{Q^*}|}
	\ge \frac{k_{Q^*}-1}{|E_{Q^*}|} = w(Q^*)^{-1} = \|\eta^*\|_\infty.
	\end{equation*}
	Taking the minimum with respect to $\gamma\in\Gamma$ establishes~\eqref{eq:necessary-and-sufficient}.  The first equality in~\eqref{eq:game-value} then follows from~\eqref{eq:eta-max-as-game-value}.
	\end{proof}

	\begin{remark}
	Although it was not formulated in the present game-theoretic notation, Cunningham~\cite{cunningham1985optimal} studied the quantity $w(Q^*)$ in~\eqref{eq:max-eta-star} (called the \emph{strength} of $G$) and provided an algorithm for finding $Q^*$.
	\end{remark}
	
	\subsection{\texorpdfstring{$1$}{1}-modulus versus \texorpdfstring{$2$}{2}-modulus}
	\label{sec:1-vs-2}
	
	Above, we have given a method for solving the broadcast game using the solution
	of a $2$-modulus problem.  The solution given in~\cite[Theorem~2]{gueye2010design}, on the other hand, is equivalent to the solution of the corresponding $1$-modulus problem.  Indeed, the Lagrangian dual to the $1$-modulus problem is (see~\cite[Equation~(7)]{albin2015modulus})
	\begin{equation} \label{eq:1modprob2}
	\begin{split}
	\text{maximize} &\hspace{8mm} \textbf{1}^T\lambda, \\
	\text{subject to} &\hspace{8mm} \mathcal{N}^T\lambda \le \textbf{1}, \lambda \ge \textbf{0}.
	\end{split}
	\end{equation}
	The connection between~\eqref{eq:1modprob2} and~\cite[Equation~(16)]{gueye2010design} can be made through the probabilistic interpretation of modulus~\cite{albin2016minimal}.  If $\lambda$ in~\eqref{eq:1modprob2} is split as $\lambda=\nu\mu$ with $\nu \ge 0$ and $\mu \in \mathcal{P}(\Gamma)$, then an optimal $\lambda^*$ satisfies $\lambda^*=w(Q^*)\mu_\infty^*$, where $\mu_\infty^*$ is any pmf satisfying $\max\limits_{e\in E}\mathbb{E}_{\mu_\infty^*}[\mathcal{N}(\underline{\gamma},e)]=w(Q^*)^{-1}$. This implies that any such optimal $\mu_\infty^*$ is a solution to~\cite[Equation~(16)]{gueye2010design}.  Thus, although~\cite{gueye2010design} does not explicitly use the vocabulary of modulus, we will refer to this method of solving the game as the $1$-modulus approach.
	
	While conceptually similar, the $1$-modulus and $2$-modulus approaches to the game have some significant differences.  The $1$-modulus approach involves first finding a \emph{critical subset of edges} (essentially, finding $Q^*$).  This provides the value of the game as well as an optimal strategy for \pE.  Although not explicitly described in the paper, an optimal strategy for \pB would presumably then be found by solving the linear program~\cite[Equation~(16)]{gueye2010design}.
	Recently in~\cite{szeszler2017security} the value of the game and strategies of both players were discussed in a more general matroid setting using a different approach. 
	
	A $2$-modulus approach to solving the game, on the other hand,  proceeds by first finding the optimal $\rho^*$ for~\eqref{eq:cvxp} by using an exterior-point algorithm as described in~\cite{albin2016minimal}.  The edges of $Q^*$ are then readily obtained by locating the edges on which $\rho^*$ is maximized, giving an optimal strategy for \pE.  Although the algorithm described only provides an approximation $\rho'\approx\rho^*$ to within some specified tolerance, it is nevertheless possible to determine $Q^*$ exactly by using a consequence of the \emph{deflation process} described in~\cite{albin2018fairest}.  Namely that the value $\rho^*(e)/\Mod_2(\Gamma)$ is a proper fraction on every edge with denominator bounded by $|E|$.  Thus, by selecting a suitably small tolerance, one can determine on exactly which edges $\rho^*$ is maximized by considering the edges on which $\rho'$ is approximately maximized.
	
	Moreover, the optimal dual variables $\lambda^*$ can be rescaled to produce an optimal $\mu^*$ as described above, thus giving an optimal strategy for \pB with no additional work.  This is particularly effective on larger graphs with large numbers of spanning trees; by using an exterior-point approach, one typically only needs to explore a tiny fraction of the constraints in order to get an accurate solution.
	
	Another interesting aspect of the difference between the $1$-modulus and $2$-modulus approaches to the game is the fact that the $2$-modulus problem is more restrictive; a $2$-modulus solution will also be a $1$-modulus solution, but not the other way around.  A simple illustration of this idea can be seen in the following example.
	\begin{figure}
		\begin{center}
		    \includegraphics[width=0.5\textwidth]{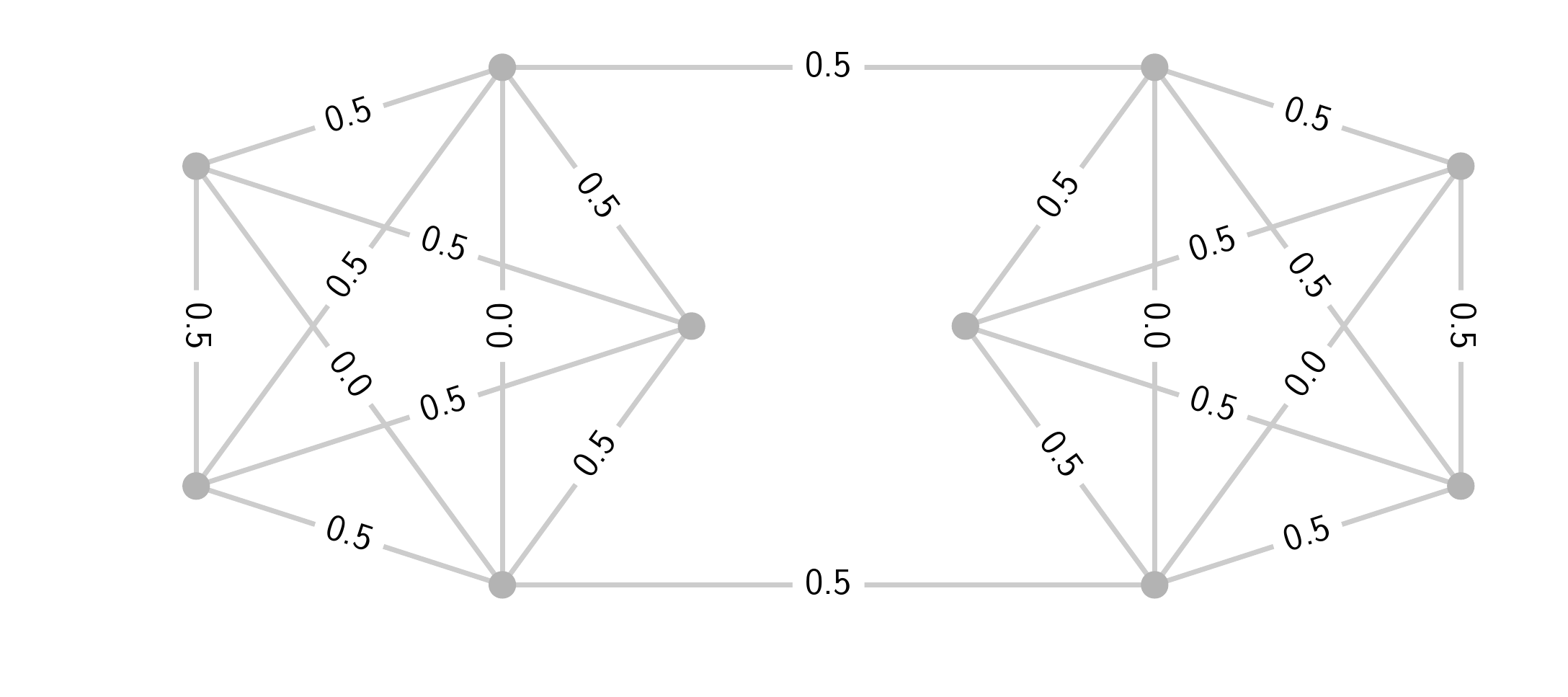}%
			\includegraphics[width=0.5\textwidth]{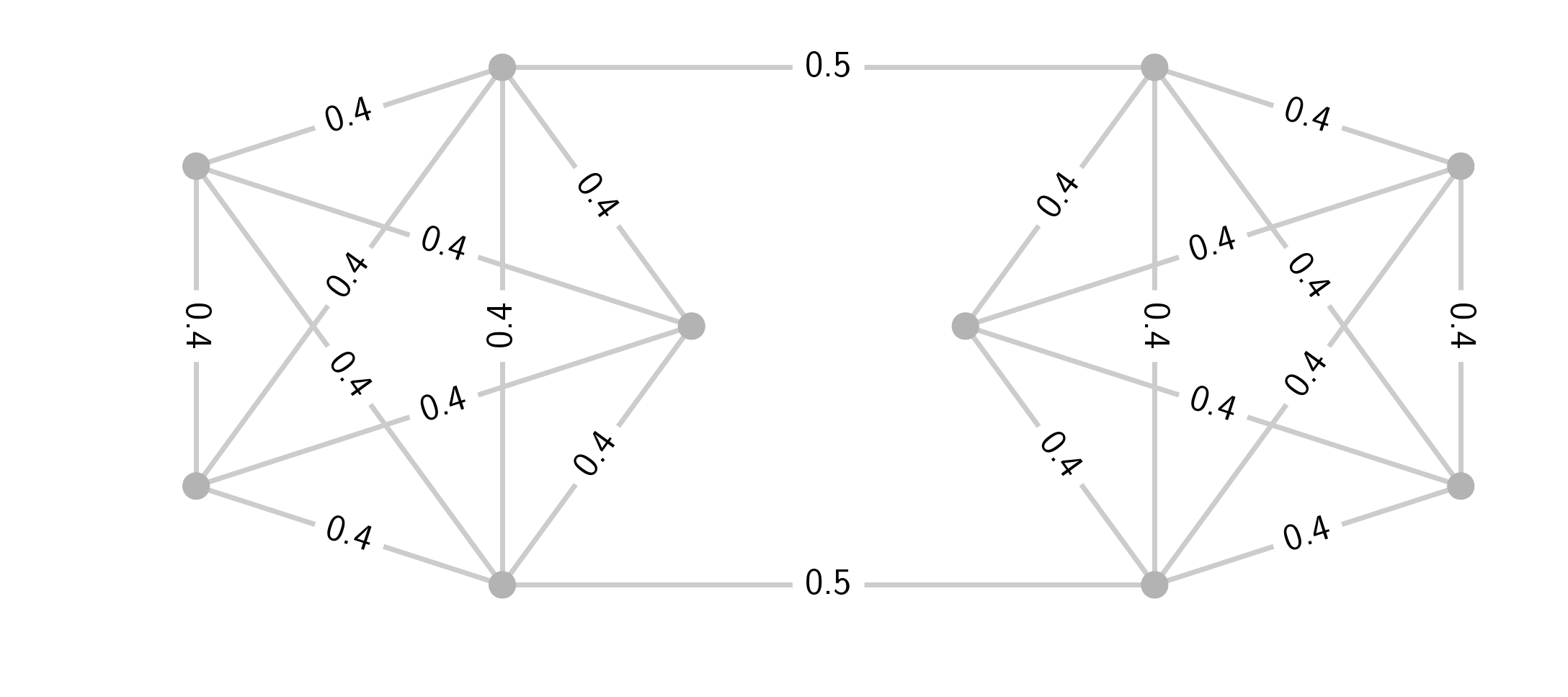}
			\caption{On the left, one possible choice of $\eta_\infty^*=\mathcal{N}^T\mu_\infty^*$ for the $1$-modulus problem, on the right, the unique $\eta^*=\mathcal{N}^T\mu^*$ for the $2$-modulus problem.}
			\label{fig:2-mod-example}
		\end{center}
	\end{figure}
	
	Consider the graph shown in Figure~\ref{fig:2-mod-example}, formed by connecting two copies of $K_5$ by two bridges.  Since at least one of the two bridges must be used in any spanning tree, the value of the secure broadcast game must be at least 1/2.  An optimal strategy for \pB can be constructed as follows.  Let $T_1$ and $T_2$ be the disjoint spanning trees of $G$ shown in Figure~\ref{fig:2-mod-example-trees}. If $\mu$ is the pmf that chooses between these two trees with 1/2 probability each, then the maximum expected edge usage is 1/2, attained on the edges of $\gamma_1\cup\gamma_2$, as shown on the left of Figure~\ref{fig:2-mod-example}.  The optimal strategy for \pE is to choose between the two bridge edges uniformly.
	
    On the other hand, consider the $2$-modulus solution.  In this case, even though the optimal pmf $\mu^*$ may not be unique, the optimal expected edge usage $\eta^*$ is.  These values are shown on the right of Figure~\ref{fig:2-mod-example}.  Observe that, following this strategy, \pB uses each non-bridge edge only 2/5 of the time.
    
    The difference between the two solutions can be understood through the probabilistic interpretation of modulus.  The $1$-modulus approach corresponds to finding a pmf $\mu$ and corresponding expected edge usage vector $\eta=\mathcal{N}^T\mu$ that minimizes $\|\eta\|_\infty$.  The $2$-modulus approach, on the other hand, corresponds to minimizing the variance of $\eta$ (see~\cite{albin2018fairest}).  Both options shown in Figure~\ref{fig:2-mod-example} minimize the maximum value, but only the one on the right minimizes the variance.
	
	\begin{figure}
	\begin{center}
	    \includegraphics[width=0.5\textwidth]{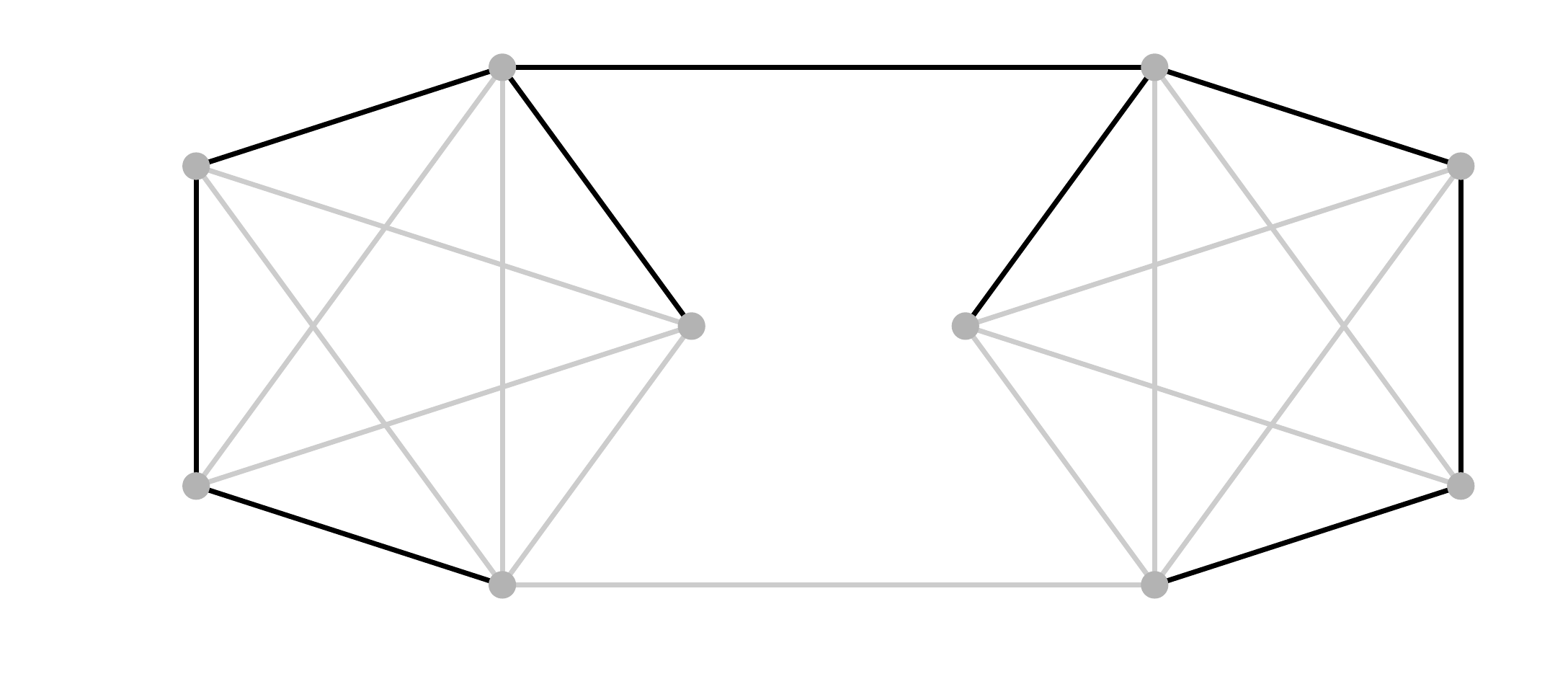}%
		\includegraphics[width=0.5\textwidth]{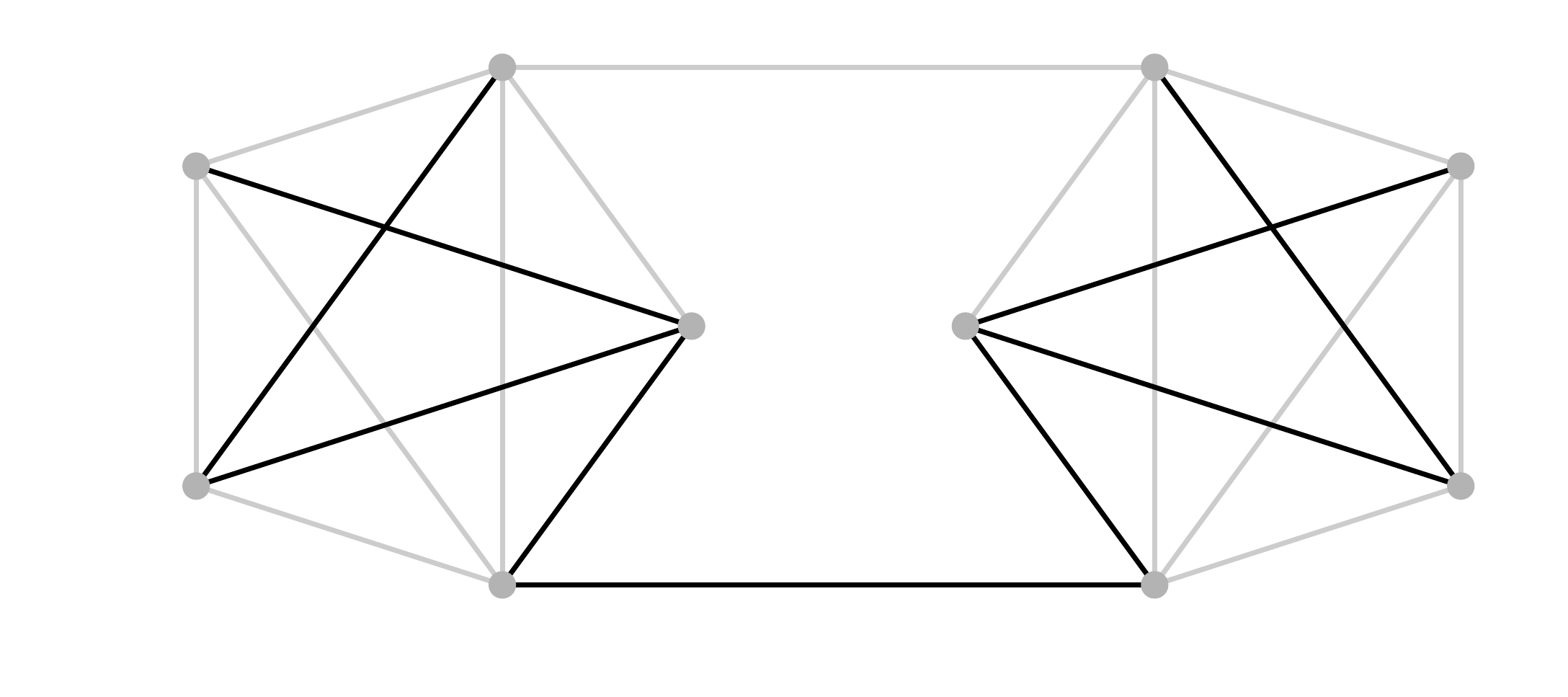}
		\caption{A pair of disjoint spanning trees, $T_1$ and $T_2$.}
		\label{fig:2-mod-example-trees}
	\end{center}
    \end{figure}
	
	\section{Examples}\label{sec:algorithms-and-examples}
	
	In this section, we demonstrate the connection between modulus and the broadcast game with a few illustrative examples.

	\subsection{The house graph}

	As a simple example of the theory developed above, consider the broadcast
	game on the house graph shown in Figure~\ref{fig:house-spanning-trees}. We shall
	demonstrate how the theory of modulus can be used to solve the secure
	broadcast game, using the theory developed in~\cite{albin2018fairest}.  We begin with a
	definition of the concept of a homogeneous graph.

	\begin{definition}
	A graph $G$ is said to be \emph{homogeneous} (with respect to spanning tree modulus) if the unique edge usage probability vector $\eta^*=\mathcal{N}^T\mu^*$ is constant (i.e., parallel to the vector $\mathbf{1}$).  A graph that is not homogeneous is called \emph{nonhomogeneous}.
	\end{definition}

	Homogeneous graphs play a special role in the theory of spanning tree modulus; they essentially form a collection of ``atoms'' into which any graph can be decomposed.  The reader is directed to~\cite{albin2018fairest} for details.  Here, we reproduce one important theorem~\cite[Corollary 4.5]{albin2018fairest}

	\begin{theorem}\label{thm:homogeneous}
	A graph $G$ is homogeneous if an only if there exists a pmf $\mu\in\mathcal{P}(\Gamma)$ so that $\mathbb{P}_\mu(e\in\underline{\gamma})$ is independent of $e\in E$.  If such a pmf exists, it is optimal for the MEO problem~\eqref{eq:MEO-opt} and
	\begin{equation*}
	\mathbb{P}_\mu(e\in\underline{\gamma}) = \frac{|V|-1}{|E|}.
	\end{equation*}
	\end{theorem}

	To see how this theorem applies in the case of the house, first note that the house has 5 vertices and 6 edges, so $(|V|-1)/|E|=2/3$.  If we can find a pmf $\mu$ that gives every edge a $2/3$ probability of inclusion in the random tree $\underline{\gamma}$, then Theorem~\ref{thm:homogeneous} proves the optimality of $\mu$.  In fact, there are infinitely many such $\mu$ in this case.  For example, using the enumeration of spanning trees in Figure~\ref{fig:house-spanning-trees}, consider the pmf $\mu$ that selects uniformly among the three trees $\gamma_1$, $\gamma_9$ and $\gamma_{10}$.  It is straightforward to verify that each edge appears in exactly two of these three trees and, therefore, that $\mathbb{P}_\mu(e\in\underline{\gamma})=2/3$ for this choice of $\mu$.

	\begin{figure}
		\begin{center}
			\includegraphics[width=\textwidth]{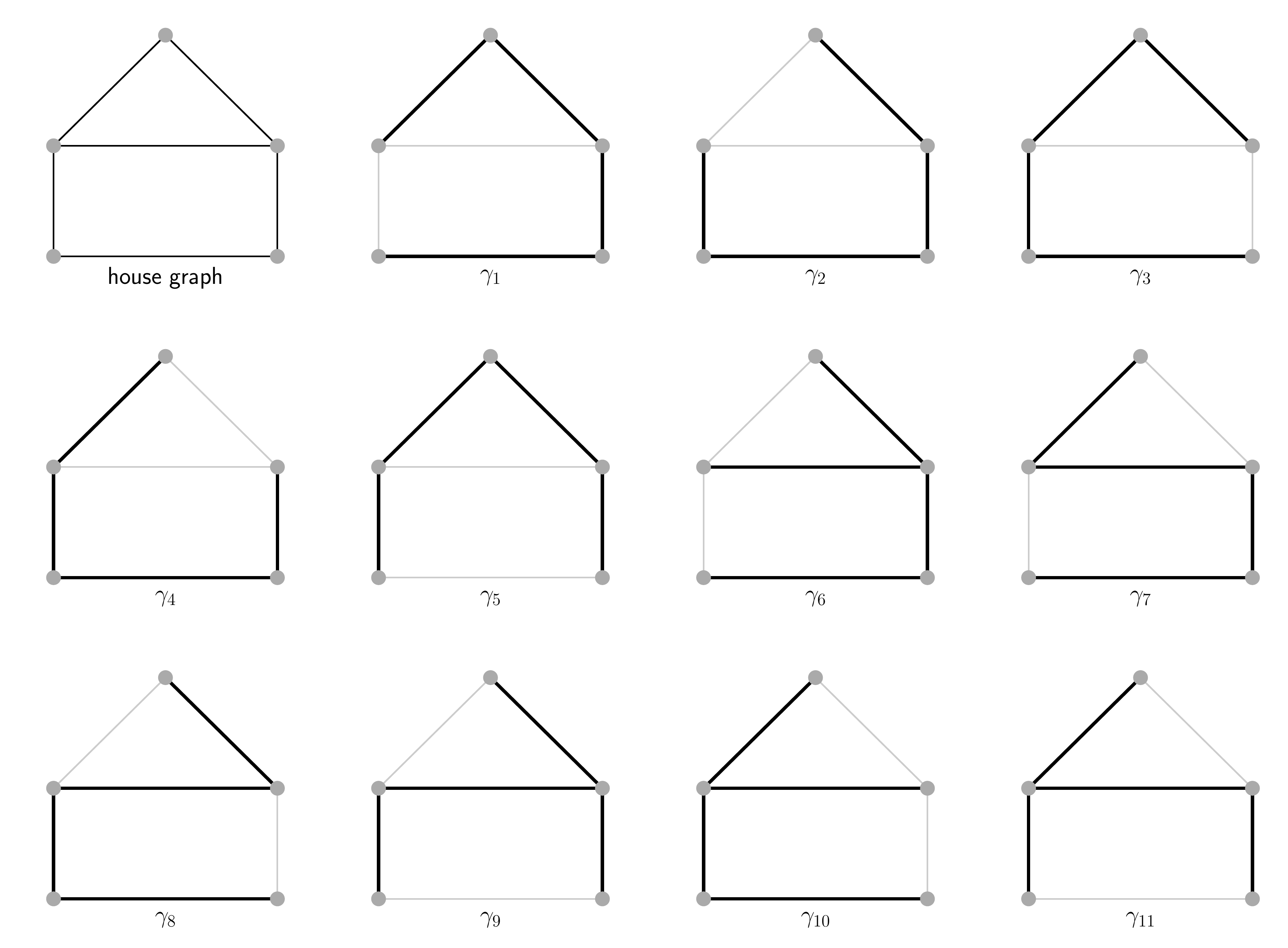}
			\caption{The house graph and its spanning trees.}
			\label{fig:house-spanning-trees}
		\end{center}
	\end{figure}

	Once one optimal pmf $\mu^*$ is found, the set of all optimal pmfs can be expressed as the polyhedron
	\begin{equation*}
	\left\{z\in\mathbb{R}^{\Gamma} : \mathcal{N}^Tz = 0,\quad \mathbf{1}^Tz = 0,\quad \mu^* + z \ge 0 \right\}.
	\end{equation*}
	For the house, this polyhedron can be shown to have six vertices, each of the form
	$\mu = \frac{1}{3}\left(\delta_{\gamma_i}+\delta_{\gamma_j}+\delta_{\gamma_k}\right)$.  In other words, the polyhedron of optimal pmfs is generated by six pmfs, each of which is uniformly distributed on a set of three spanning trees.  These six sets are
	\begin{equation*}
	\{  \gamma_{1},  \gamma_{9},  \gamma_{10}  \},
	\{  \gamma_{1},  \gamma_{8},  \gamma_{11}  \},
	\{  \gamma_{5},  \gamma_{7},  \gamma_{8}  \},
	\{  \gamma_{3},  \gamma_{6},  \gamma_{11}  \},
	\{  \gamma_{3},  \gamma_{7},  \gamma_{9}  \},\text{ and } 
	\{  \gamma_{5},  \gamma_{6},  \gamma_{10}  \}.
	\end{equation*}

	It is interesting to note that two spanning trees, $\gamma_2$ and $\gamma_4$, are absent from these sets.  In other words, neither $\gamma_2$ nor $\gamma_4$ is in the support of any optimal pmf.  These two trees are examples of the \emph{forbidden trees} defined in~\cite{albin2018fairest}.  There are several ways to see why such forbidden trees might exist in general.  Here, we provide a proof specific to the house graph.

	\begin{theorem}
	If $\mu\in\mathcal{P}(\Gamma)$ is optimal for the MEO problem~\eqref{eq:MEO-opt} on the house graph, then
	\begin{equation*}
	\supp\mu\subset \Gamma\setminus\{\gamma_2,\gamma_4\},
	\end{equation*}
	using the spanning tree enumeration given in Figure~\ref{fig:house-spanning-trees}.
	\end{theorem}

	\begin{proof}
	As observed above, the pmf $\mu=\frac{1}{3}\left(\delta_{\gamma_{1}}+\delta_{\gamma_{9}}+\delta_{\gamma_{10}}\right)$ has the property that $\mathbb{P}_\mu(e\in\underline{\gamma})\equiv 2/3$.  By Theorem~\ref{thm:homogeneous}, then, the house graph is homogeneous and $\mu$ is optimal for the MEO problem.  By the uniqueness of $\eta^*$ in~\eqref{eta*}, it follows that \emph{every} optimal pmf has uniform edge usage probabilities equal to 2/3. Let $\mu$ be any optimal pmf and let $E'\subset E$ be the top three edges of the house (the roof and ceiling).  Observe that
	\begin{equation*}
	|\gamma\cap E'| =
	\begin{cases}
	1 & \text{if }\gamma\in\{\gamma_2,\gamma_4\},\\
	2 & \text{if }\gamma\in\Gamma\setminus\{\gamma_2,\gamma_4\}.
	\end{cases}
	\end{equation*}
	So,
	\begin{equation*}
	\begin{split}
     3\cdot \frac{2}{3} &= 2 =\sum_{e\in E'}\eta^*(e) = \sum_{e\in E'}\sum_{\gamma\in\Gamma}\mathcal{N}(\gamma,e)\mu(\gamma)
	= \sum_{\gamma\in\Gamma}\mu(\gamma)\sum_{e\in E'}\mathcal{N}(\gamma,e)
	= \sum_{\gamma\in\Gamma}\mu(\gamma)|\gamma\cap E'|\\
	&= 2\bigg(\sum_{\gamma\in\Gamma}\mu(\gamma)\bigg) - \mu(\gamma_2) - \mu(\gamma_4)
	= 2 - \mu(\gamma_2) - \mu(\gamma_4).
	\end{split}
	\end{equation*}
	This is only possible if $\mu(\gamma_2)=\mu(\gamma_4)=0$.
	\end{proof}

	\subsection{Other homogeneous graphs}

	A consequence of the connection between modulus and the broadcast game is that the value of the game on a homogeneous graph is $(|V|-1)/|E|$.  This is interesting because homogeneity can sometimes be established without an explicit construction of an optimal pmf.  Take, for example, the complete graph $K_n$.  From the symmetry of the graph, it is straightforward to show that the optimal density $\rho^*$ for the modulus problem is uniform and, therefore, so is $\eta^*$.  While it is possible to construct an optimal pmf for $K_n$ (for example, choose a vertex uniformly at random and take the star of edges incident upon this vertex), we can use homogeneity to establish that the value of the broadcast game on $K_n$ is $2/n$.

	Similarly, consider an arbitrary connected $d$-regular graph.  In~\cite{albin2018fairest}, it was shown that almost every such graph is homogeneous.  Thus, it follows that the value of the broadcast game on almost every $d$-regular graph is $\frac{2(|V|-1)}{d|V|}$.  Unlike in the case of $K_n$, it is not as immediately evident how to construct an optimal pmf for an arbitrary $d$-regular graph.
	
	\subsection{The effect of adding edges}

    As a final example, we consider the effect that adding edges has on the value of the game.  To begin, consider the graph shown in the upper left corner of Figure~\ref{fig:modulus-example}, formed by connecting a copy of $K_5$ to a copy of $K_6$ by a single bridge.
    \begin{figure}
		\begin{center}
			\includegraphics[width=\textwidth]{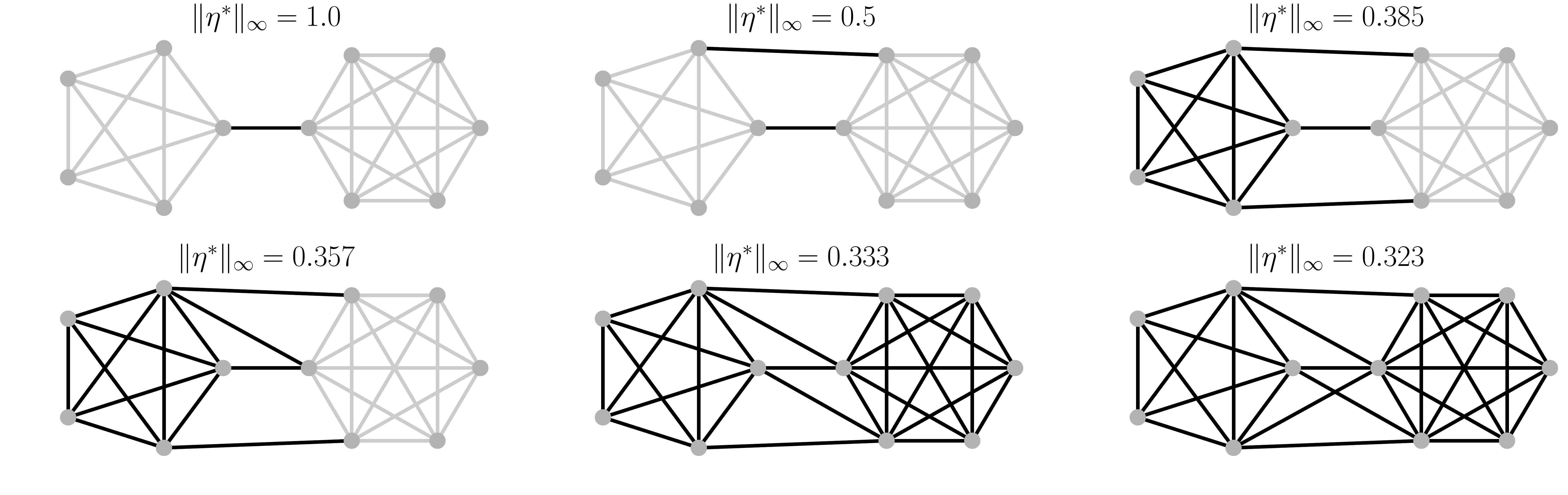}
			\caption{The maximum value of $\eta^*$ (associated with $2$-modulus) for a sequence of graphs.  The darkened edges show where $\eta^*$ attains its maximum.}
			\label{fig:modulus-example}
		\end{center}
	\end{figure}
	This is an example of a $1$-connected graph considered in Theorem~\ref{thm:pure-strategy}.  Since every spanning tree must use the bridge, the value of the game is 1 (\pE always wins).
	
	The top middle portion of the figure shows the effect of adding an additional bridge between the two complete components.  In this case, the broadcaster is forced to use at least one of these two bridges.  The solution to the game involves the broadcaster choosing trees that use \emph{exactly} one of the bridges in such a way that each bridge is used half the time.  The eavesdropper chooses among the two edges uniformly, giving a 50\% chance that the eavesdropper will intercept the broadcast.
	
	One might expect that adding a third bridge (the top right portion of the figure) would result in a game with value $1/3$ (the broadcaster chooses trees that use exactly one of the three bridges and the eavesdropper chooses uniformly among the three edges).  However, this pair of strategies is not optimal.  Indeed, if \pB were to choose this strategy, then every spanning tree selected by \pB would restrict as a spanning tree to the $K_5$ component and, therefore, would use 4 of the 10 edges there.  So \pE could win 40\% of the time by choosing uniformly among the edges of the $K_5$ component.  In the solution shown, the broadcaster is able to choose a random tree that performs better than this.  The eavesdropper is forced to choose uniformly among a larger set of edges (the 13 darker edges) and will win only 38.5\% of the time.  In this case, the minimum feasible partition consists of placing all nodes of the $K_6$ component into a single part and separating each node of $K_5$ into its own part.  This gives a partition $Q$ with $k_Q=6$ and $|E_Q|=13$, so $w(Q)^{-1}=5/13\approx 0.3846$. 	When run on this example, the modulus algorithm from~\cite{albin2016minimal} produces an optimal pmf $\mu^*$  supported on 23 trees. An optimal strategy for $\pB$, therefore, is to choose from among those 23 spanning trees according to the pmf $\mu^*$. Figure~\ref{fig:optimal-support} shows the first 9 spanning trees ordered according to the value of optimal pmf $\mu^*$.  Observe that the spanning trees shown always restrict as spanning trees of the $K_6$ component.  This is a consequence of the fact that the $K_6$ subgraph, in this case, is a homogeneous core as described in~\cite{albin2018fairest}.  On the other hand, the trees in the figure do not, in general, restrict as trees of the $K_5$ subgraph (which would require only one of the three bridge edges to be present).  Indeed, the third row shows spanning trees that use two of the three bridges.  Some of the lower-probability trees (not shown in the figure) even use all three bridges.
	
	\begin{figure}
		\begin{center}
			\includegraphics[width=\textwidth]{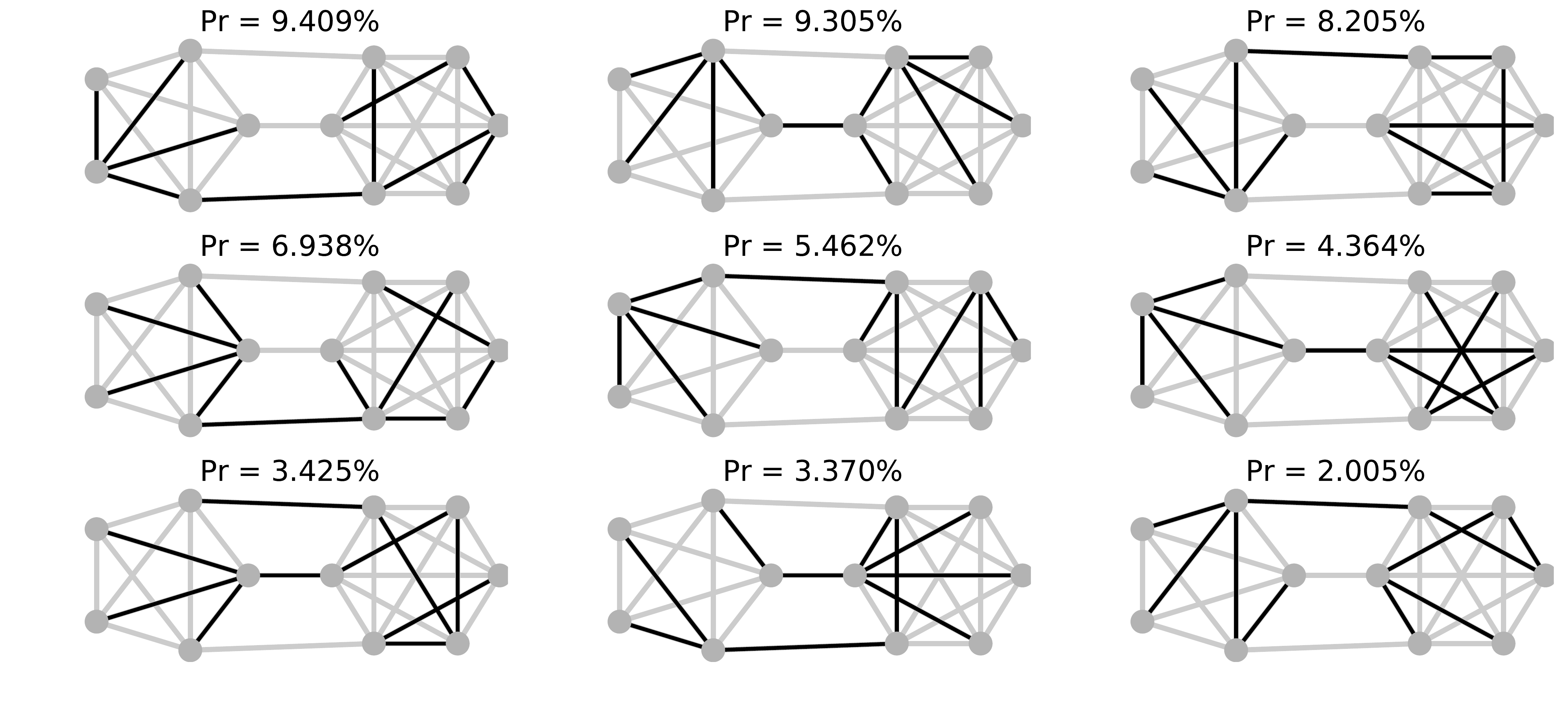}
			\caption{Nine spanning trees with largest value of $\mu^*$.}
			\label{fig:optimal-support}
		\end{center}
	\end{figure}
	
	Adding a fourth bridge has a similar effect (bottom left portion of the figure).  Upon adding a fifth bridge, the graph become homogeneous; the minimum feasible partition is the trivial one, separating each node into its own part.  For this partition $k_Q=11$, $|E_Q|=30$ and $w(Q)^{-1}=1/3$.

	For comparison we have shown $\|\eta_0\|_{\infty}$ from the uniform pmf $\mu_0$ in Figure~\ref{fig:uniform-example}.  Again, we start with the graph in the upper left corner of Figure~\ref{fig:uniform-example}. Since every spanning tree has to use the bridge the expected edge usage is 1 on that edge, and \pE wins by listening there.  When a second edge is added, a broadcaster choosing uniformly among spanning trees will overuse the two bridges, sometimes picking trees that use both.  (Compare this with the optimal strategy for \pB described above.)  By the time a third bridge is added, the difference between the uniform and optimal strategies is striking.  The uniform broadcaster uses each of the three bridges over half the time while the optimal broadcaster, by optimally using the minimum feasible partition, 
	forces the eavesdropper to choose among 13 edges, each used no more than 38.5\% of the time.
	
	A key observation here is that choosing uniformly among spanning trees does not guarantee an even distribution of edge usage.  Consider, for example, the bottom right portion of the figure.  From the previous figure, we know it is possible for \pB to ensure that all edges are equally likely to occur.  However, the uniform distribution uses one edge (the far left one) more than any other edge.   
	
	\begin{figure}
		\begin{center}
			\includegraphics[width=\textwidth]{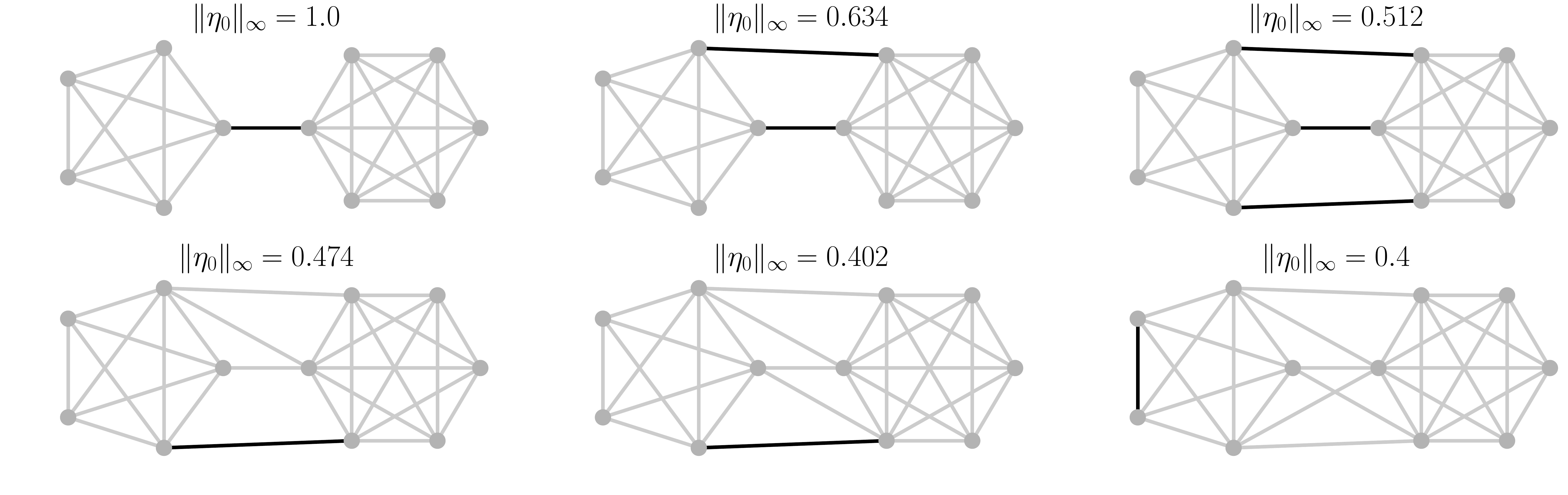}
			\caption{The maximum value of $\eta_0$ (associated with the uniform pmf $\mu_0$) for a sequence of graphs.  The darkened edges show where $\eta_0$ attains its maximum.}
			\label{fig:uniform-example}
		\end{center}
	\end{figure}

    \bibliographystyle{acm}
    \bibliography{refs}

\end{document}